\documentclass[reqno,11pt]{amsart}
\usepackage{amsmath,amsfonts,amsthm,amssymb,amsbsy,upref,color,graphicx,amscd,hyperref,enumerate}
\usepackage[active]{srcltx}
\usepackage[latin1]{inputenc}
\usepackage{bbm}
\usepackage{algorithm}
\usepackage{algorithmicx}
\usepackage{algpseudocode}
\usepackage{subfig} 
\usepackage{pgf}
\usepackage{xfrac}
\usepackage[a4paper,margin=2.96truecm]{geometry}

\DeclareMathOperator{\dive}{div}

\def\ds{\displaystyle}
\def\eps{{\varepsilon}}

\def\O{\Omega}
\def\R{\mathbb{R}}

\def\A{\mathcal{A}}
\def\E{\mathcal{E}}
\def\F{\mathcal{F}}
\def\HH{\mathcal{H}}

\def\M{\mathcal{M}}

\def\Dr{D}

\newcommand{\be}{\begin{equation}}
\newcommand{\ee}{\end{equation}}
\newcommand{\bib}[4]{\bibitem{#1}{\sc#2: }{\it#3. }{#4.}}
\newcommand{\cp}{\mathop{\rm cap}\nolimits}

\newcommand{\ind}{\mathbbm{1}}

\numberwithin{equation}{section}
\theoremstyle{plain}
\newtheorem{theo}{Theorem}[section]
\newtheorem{lemma}[theo]{Lemma}
\newtheorem{cor}[theo]{Corollary}
\newtheorem{prop}[theo]{Proposition}

\theoremstyle{remark}
\newtheorem{rema}[theo]{Remark}
\newtheorem{exam}[theo]{Example}

\newenvironment{ack}{{\bf Acknowledgements. }}

\title[A shape optimal control problem and its probabilistic counterpart]{A shape optimal control problem\\ and its probabilistic counterpart}

\author{Giuseppe Buttazzo, Bozhidar Velichkov}

\begin{document}

\maketitle

\begin{abstract}
In this paper we consider a shape optimization problem in which the data in the cost functional and in the state equation may change sign, and so no monotonicity assumption is satisfied. Nevertheless, we are able to prove that an optimal domain exists. We also deduce some necessary conditions of optimality for the optimal domain. The results are applied to show the existence of an optimal domain in the case where the cost functional is completely identified, while the right-hand side in the state equation is only known up to a probability $P$ in the space $L^2(D)$.
\end{abstract}



\section{Introduction}\label{sintro}

In this paper we consider a shape optimization problem of the form
$$\min\big\{F(\O)\ :\ \O\in\A\big\}$$
where $F$ is the shape cost function and $\A$ the class of admissible domains. For this kind of problems in general one should not expect the existence of an optimal domain, since minimizing sequences could be made of finely perforated domains, leading at the limit to existence of only relaxed solutions that are not domains but Borel measures. In some particular cases however an optimal domain exists; the most general existence result providing optimal solutions that are domains and not measures is still given by Theorem 2.5 of \cite{budm93} (see also \cite{bdm91}), where the crucial assumption is that the shape cost functional $F$ is monotone decreasing with respect to the set inclusion. A similar result for monotone costs in the framework of optimization problems for Schr\"odinger potentials has been obtained in \cite{bgrv14}.

The cost functional $F$ we consider here is not in general monotone decreasing for the set inclusion; nevertheless we are able to prove the existence of an optimal domain for it. 
We fix:
\begin{itemize}
\item a bounded Lipschitz domain $D\subset\R^d$,
\item a right-hand side $f\in L^2(D)$,
\item a cost coefficient $g\in L^2(D)$,
\end{itemize}
and we consider the admissible class of domains
\be\label{admclass}
\A=\big\{\O\subset D,\ \O\hbox{ quasi open, }|\O|\le1\big\},
\ee
where $|\cdot|$ denotes the Lebesgue measure in $\R^d$. In order the problem be nontrivial we assume that $|D|>1$.

\subsection{Statement of the problem and main results}\label{sstatem}

For every $\O\in\A$ we denote by $u_\O$ the unique solution of the Dirichlet problem
\be\label{pde}
-\Delta u=f\quad\hbox{in}\quad\O,\qquad u\in H^1_0(\O),
\ee
where $H^1_0(\O)$ is the Sobolev space of functions in $H^1(\R^d)$ vanishing capacity quasi everywhere outside $\O$. The optimization problem we are dealing with is
\be\label{minpb}
\min\left\{\int_D g(x)u_\O(x)\,dx\ :\ \O\in\A\right\}.
\ee
Note that, by the definition of $u_\O$, problem \eqref{minpb} is an optimal control problem, where $H^1_0(D)$ is the space of states, $\A$ is the set of controls, \eqref{pde} is the state equation, and $\int_D g(x)u_\O(x)\,dx$ is the cost function. We stress the fact that we do not assume any sign condition on the data $f,g$.

It is well known that in the special case $g=-f/2$ the optimization problem \eqref{minpb} can be written, through an Euler-Lagrange derivation and an integration by parts, as
$$\min\big\{\E(\O)\ :\ \O\in\A\big\}$$
where $\E(\O)$ is the Dirichlet energy
$$\E(\O)=\min\left\{\int\Big[\frac12|\nabla u|^2-f(x)u\Big]\,dx\ :\ u\in H^1_0(\O)\right\}.$$
This would allow to see easily, thanks to the inclusion of the Sobolev spaces
$$\O_1\subset\O_2\ \Longrightarrow\ H^1_0(\O_1)\subset H^1_0(\O_2),$$
that the shape function $\E(\O)$ is decreasing with respect to the set inclusion, and then an immediate application of the existence Theorem 2.5 of \cite{budm93} would give a solution $\O_{opt}$ of problem \eqref{minpb}, with the additional property that $|\O_{opt}|=1$.

The same conclusion would easily hold when $f\ge0$ and $g\le0$; indeed, in this case, thanks to the maximum principle, the solutions $u_\O$ would be monotonically increasing with respect to $\O$, and again the shape cost function $\O\mapsto\int_D g(x)u_\O(x)\,dx$ would turn out to be decreasing with respect to $\O$, providing then (again by the existence Theorem 2.5 of \cite{budm93}) an optimal solution $\O_{opt}$ of problem \eqref{minpb}, with $|\O_{opt}|=1$.

On the contrary, when $f$ and $g$ are general functions in $L^2(D)$, the existence Theorem 2.5 of \cite{budm93} cannot be applied and the existence of an optimal domain for the minimization problem \eqref{minpb} requires a deeper investigation. Our main existence result is the following.

\begin{theo}\label{t:existence}
Let $f,g\in L^2(D)$ be given; then the minimization problem \eqref{minpb} admits a solution $\O_{opt}$ in the admissible class $\A$.
\end{theo}
Moreover we prove that 
\begin{itemize}
\item if $g\ge0$ we have either $|\O_{opt}|=1$ or $|\O_{opt}|<1$ and $\{f<0\}\subset\O_{opt}$ (Theorem \ref{p:contain}); similarly, if $f\ge0$ we have either $|\O_{opt}|=1$ or $|\O_{opt}|<1$ and $\{g<0\}\subset\O_{opt}$;
\item if $\O_{opt}$ is smooth, the state functions $u$ and $v$ on $\O_{opt}$, corresponding to the solutions of the PDE \eqref{pde} with right-hand side $f$ and $g$ respectively, satisfy 
$$\frac{\partial u}{\partial n}\frac{\partial v}{\partial n}=const\quad\text{on}\quad \partial\O_{opt}\cap D,$$
the constant being zero if $|\O_{opt}|<1$ (Subsection \ref{sub:optcond});
\item if $|\O_{opt}|<1$ and $f\ge 0$, then the function $v_\O$, corresponding to the function $g$, is a solution of an obstacle problem (Proposition \ref{p:ostacolo}) and thus, under some appropriate assumptions on the regularity of $g$, the optimal set $\O_{opt}$ is open and its boundary is smooth (Corollary \ref{cor:ostacolo});
\item if $D=\R^d$ and $f,g$ are radially symmetric functions, $f$ radially decreasing and $g$ radially increasing, then the optimal set $\O_{opt}$ is a ball centered in zero (Proposition \ref{radprop}).
\end{itemize}

\subsection{A stochastic optimal control problem}\label{sstoch}

A probabilistic counterpart of the optimization problem \eqref{minpb} is given by the case when the function $g$ appearing in the cost functional \eqref{minpb} is completely known, while the right-hand side $f$ in \eqref{pde} has the form $f=f_0+h$, where $f_0$ is given and $h$ is some random perturbation. The purpose of such a model is to obtain shapes corresponding to mechanical structures that are robust and reliable even if the data are not completely known. Several models involving uncertainties has been already studied; from the numerical point of view we refer for instance to \cite{all15} and the references therein, while in most of the cases there are no available theoretical results, even in some simplified situations.

An interesting result in this spirit is concerned with the existence of optimal domains for the worst-case functional 
$$\min_{\O\in\A}\,\sup_{h\in L^p}\int_D g(x)R_\O(f_0+h)\,dx$$
and was proved in \cite{bebuve} under the assumptions that $g\le 0$, $f_0>0$, and the perturbation $h$ is small. Here $R_\O$ denotes the resolvent operator which associates to every $f\in L^2(D)$ the solution $u_\O$ of \eqref{pde}.

Another situation of practical interest is when the perturbation $h$ belongs to some probability space and the cost functional is given by the average over all possible choices of $h$. The existence of minimizer in this situation can be deduced from Theorem \ref{t:existence} without any smallness assumption on the incertainty $h$.

More precisely, given a probability $P$ on $L^2(D)$, our goal is to minimize the averaged cost
\be\label{avecost}
F(\O)=\int\left(\int_D g(x)R_\O(f)\,dx\right)\,dP(f)
\ee
over the admissible class $\A$ given by \eqref{admclass}. We assume that the barycenter $B_P$ of $P$, given by
$$B_P=\int f\,dP(f)$$
belongs to $L^2(D)$. We notice that $B_P$ is well defined when $P$ is such that
$$\int\|f\|_{L^2}\,dP(f)<+\infty\;.$$
Thus, using the fact that the resolvent operator $R_\O$ is self-adjoint, the cost functional in \eqref{avecost} can be written as
$$F(\O)=\int\left(\int_D R_\O(g)f\,dx\right)\,dP(f)=\int_D R_\O(g)B_P(x)\,dx=\int_D g(x)R_\O(B_P)\,dx\;,$$
and we are then in the framework of the existence Theorem \ref{t:existence}.

\subsection{Organization of the paper}

In Section \ref{sexist} we prove the existence of an optimal domain $\O_{opt}$ (Theorem \ref{t:existence}). The study of the regularity properties of the optimal domains is an interesting and difficult issue; in Subsection \ref{sub:optcond} we compute the so called {\it shape derivative} assuming that $\O_{opt}$ is regular enough. Obtaining the regularity of a general solution $\O_{opt}$ from its minimality would be a very interesting result.

In Section \ref{s:nonsat} we study the minimizers for which the constraint $|\O_{opt}|\le1$ is not saturated. Note that this is a rather general situation, since no monotonicity of the shape cost function is assumed. Nevertheless, in several cases ($f\ge0$ and $|\{g<0\}|\ge1$) we may still obtain that the optimal domain verifies $|\O_{opt}|=1$ as we see in Theorem \ref{p:contain}. In Section \ref{sub:ostacolo} we show that $\O_{opt}$ is a solution of an obstacle problem and as a consequence we obtain that it has a regular free boundary in the sense of Corollary \ref{cor:ostacolo}. 

Finally, in Section \ref{s:radial} we study the case of radially symmetric functions $f$ and $g$. It is natural to expect that under this assumption the optimal domains are balls centered at zero. Also in this case the lack of monotonicity of the functional represents a difficult issue since the energy does not necessarily decrease under symmetrization. Nevertheless, we are able to prove that for every $\O$ there is a ball $B$ (not necessarily of the same measure as $\O$) having a smaller energy. We also provide an example of an optimal set $\O_{opt}$ of measure strictly smaller than one.

\section{Sobolev spaces, quasi-open sets and capacitary measures}

In this section we briefly recall several notions related to capacity theory, quasi open sets, and capacitary measures; we refer to the book \cite{bubu05} for more details concerning these notions.

\subsection{Sobolev functions and their representatives}

The Sobolev space $H^1(\R^d)$ is the closure of $C^\infty_c(\R^d)$ with respect to the norm 
$$\|u\|_{H^1}=\left(\int_{\R^d}|\nabla u|^2\,dx+\int_{\R^d}u^2\,dx\right)^{1/2}.$$
For every function $u\in H^1(\R^d)$ there is a set $E_u\subset\R^d$ such that: 
\begin{itemize}
\item every point in $\R^d\setminus E_u$ is a Lebesgue point for $u$, that is 
$$\ds u(x_0)=\lim_{r\to 0}\frac{1}{|B_r|}\int_{B_r(x_0)}u(x)\,dx\quad\text{for every}\quad x_0\in\R^d\setminus E_u;$$
\item $E_u$ has capacity zero, that is $\cp(E_u)=0$, where for a set $E\subset\R^d$, $\cp(E)$ is defined as
$$\cp(E):=\inf\Big\{\|\phi\|_{H^1}^2\ :\ \phi\in H^1(\R^d),\ \phi=1\text{ in a neighborhood of }E\Big\}.$$
\end{itemize}
We notice that a Sobolev function $u$ is defined up to a set of zero capacity, that is $u_1\sim u_2$ if and only if $\cp(\{u_1\neq u_2\})=0$.

\subsection{Quasi-open sets and the space $H^1_0(\O)$}

We say that a set $\O\subset\R^d$ is quasi open if it is of the form $\O=\{u>0\}$ for some $u\in H^1(\R^d)$. We notice that all the open sets are quasi-open. Given a quasi-open set $\O\subset\R^d$ we define the Sobolev space
$$H^1_0(\O):=\Big\{u\in H^1(\R^d)\ :\ \cp(\{u\ne0\}\setminus\O)=0\Big\}.$$
We notice that $H^1_0(\O)$ is a closed subspace of $H^1(\R^d)$. In fact, if $u_n\to u$ in $H^1(\R^d)$, then up to a subsequence $u_n\to u$ pointwise outside of a set of zero capacity. If $\O$ is open then $H^1_0(\O)$ coincides with the usual Sobolev space defined as the closure of $C^\infty_c(\O)$ with respect to the $H^1$ norm. Let $\O\subset\R^d$ be a quasi-open set of finite measure and let $f\in L^2(\O)$. We say that a function $u\in H^1_0(\O)$ is a solution of the equation
$$-\Delta u=f\quad\hbox{in }\O,\qquad u\in H^1_0(\O),$$
if we have
$$\int_\O\nabla u\nabla\varphi\,dx=\int_\O f\varphi\,dx\qquad\forall\varphi\in H^1_0(\O).$$

\subsection{Capacitary measures}

We say that a nonnegative Borel measure $\mu$ is {\it capacitary} if for every set $E\subset\R^d$ with $\cp(E)=0$, we have $\mu(E)=0$. We denote by $\M_{\cp}(\R^d)$ the class of capacitary measures on $\R^d$. In particular, if two functions $u_1$ and $u_2$ are in the same equivalence class of $H^1(\R^d)$, and $\mu$ is a capacitary measure, then $u_1$ and $u_2$ are in the same equivalence class of $L^2(\mu)$. For a quasi-open set $\O\subset\R^d$ and for a measure $\mu\in\M_{\cp}(\R^d)$ we define the space
$$H^1_\mu(\O)=H^1_0(\O)\cap L^2(\mu)=\Big\{u\in H^1_0(\O)\ :\ \int_{\R^d}u^2\,d\mu<\infty\Big\}.$$

\noindent For a given function $f\in L^2(\O)$ we say that $u\in H^1_\mu(\O)$ is a solution of the equation
$$-\Delta u+\mu u=f\quad\text{in }\O,\qquad u\in H^1_\mu(\O),$$
if we have
$$\int_\O\nabla u\nabla\varphi\,dx+\int_{\R^d}u\varphi\,d\mu=\int_\O f\varphi\,dx\qquad\forall\varphi\in H^1_\mu(\O).$$

\noindent Let $\mu$ be a capacitary measure in $\R^d$. The set of finiteness $\O_\mu$ of $\mu$ is defined as
$$\O_\mu=\bigcup_{u\in H^1_\mu(\R^d)}\{u\ne0\}.$$
We notice that the set $\O_\mu$ is a quasi-open set due to the fact that $H^1_\mu$ is separable. Moreover, if $\mu=0$ on $\O_\mu$, then $H^1_\mu(\R^d)=H^1_0(\O_\mu)$.

\subsection{Convergence of capacitary measures}

Consider a bounded open set $D\subset\R^d$ and the family of capacitary measures
$$\M_{\cp}(D)=\Big\{\mu\in\M_{\cp}(\R^d)\ :\ \O_\mu\subset D\Big\}.$$
For every capacitary measure $\mu\in\M_{\cp}(D)$ we consider the torsion function $w_\mu$, solution of the equation
$$-\Delta w_\mu+\mu w_\mu=1\quad\text{in }D,\qquad w_\mu\in H^1_\mu(D).$$
We notice that $w_\mu$ uniquely determines the measure $\mu$. In fact, we have
$$\O_\mu=\{w_\mu>0\}\qquad\text{and}\qquad\mu=\frac{\Delta w_\mu+1}{w_\mu}\quad\text{on }\O_\mu.$$
The set $\M_{\cp}(D)$, endowed with the distance
$$d_\gamma(\mu_1,\mu_2)=\|w_{\mu_1}-w_{\mu_2}\|_{L^2}\;,$$
is a compact metric space (see for instance \cite{dmmo87}). Moreover, the family of capacitary measures $I_\O$ associated to smooth domains $\O\subset D$ is dense in $\M_{\cp}(D)$, where the measure $I_\O$ is defined by
$$I_\O(E)=\begin{cases}
0&\text{if }\cp(E\setminus\O)=0,\\
+\infty&\text{otherwise}.
\end{cases}$$

\section{Existence of optimal shapes}\label{sexist}

In this section we prove the existence Theorem \ref{t:existence}. We first relax the problem to the class of capacitary measures $\M_{\cp}(D)$ that represents the closure of the admissible class $\A$ with respect to the $\gamma$-convergence. The relaxed problem is written again as an optimal control problem, with admissible class given by
$$\M=\big\{\mu\in\M_{\cp}(D),\ |\O_\mu|\le1\big\},$$
being $\O_\mu$ the {\it set of finiteness} of $\mu$. For every admissible $\mu\in\M$ we consider the state equation
\be\label{pdemu}
-\Delta u+\mu u=f\quad\hbox{in}\quad D,\qquad u\in H^1_0(D)\cap L^2(\mu),
\ee
and we indicate its unique solution by $u_\mu$. The relaxed optimization problem related to \eqref{minpb} can be then stated as
\be\label{relpb}
\min\left\{\int_D g(x)u_\mu\,dx\ :\ \mu\in\M\right\}.
\ee
It is convenient to introduce the {\it resolvent operator} $R_\mu:L^2(D)\to L^2(D)$ which associates to every $f\in L^2(D)$ the solution $u_\mu$ of \eqref{pdemu}. Thanks to the fact that $R_\mu$ is self-adjoint we can write the cost function as
$$\int_D g(x)\, R_\mu(f)\,dx=\int_D R_\mu(g)\,f(x)\,dx.$$

\begin{proof}[Proof of Theorem \ref{t:existence}]
It is well known that the relaxed admissible class $\M$ is compact with respect to $\gamma$-convergence and that the cost function is $\gamma$-continuous (see for instance \cite{bubu05}); therefore an optimal relaxed solution $\mu$ to problem \eqref{relpb} exists.

For every bounded continuous function $\phi$ and for every $\eps>0$ small enough we consider the capacitary measure $\mu_\eps=(1+\eps\phi)\mu$; since $\phi$ is bounded and $\eps$ is small we have that $\mu_\eps\in\M$ and $\O_{\mu_\eps}=\O_\mu$. Moreover, the spaces $H^1_{\mu_\eps}$ and $H^1_\mu$ coincide. Let us denote by $u_\eps$ the solution of the PDE
$$-\Delta u_\eps+\mu_\eps u_\eps=f\quad\hbox{in }D,\qquad u_\eps\in H^1_{\mu_\eps}$$
and by $u$ the solution of
$$-\Delta u+\mu u=f\quad\hbox{in }D,\qquad u\in H^1_\mu.$$
By the minimality of $\mu$ we have
$$\int_D gu_\eps\,dx\ge\int_D gu\,dx,$$
which gives
\be\label{weps}
\int_Dg\frac{u_\eps-u}{\eps}\,dx\ge0.
\ee
Denoting by $w_\eps$ the function $(u_\eps-u)/\eps$ we have that $w_\eps$ satisfies the PDE
$$-\Delta w_\eps+\mu w_\eps=-\phi u_\eps\mu\quad\hbox{in }D,\qquad w_\eps\in H^1_\mu.$$
Since $\mu_\eps$ $\gamma$-converges to $\mu$ we have that $u_\eps\to u$ weakly in $H^1_\mu$; hence $w_\eps\to w$ weakly in $H^1_\mu$, where $w$ is the solution of the PDE
$$-\Delta w+\mu w=-\phi u\mu\quad\hbox{in }D,\qquad w\in H^1_\mu.$$
Passing to the limit in \eqref{weps} as $\eps\to0$ gives
$$0\le\int_D gw\,dx=\int_D gR_\mu(-\phi u\mu)\,dx=-\int_D R_\mu(g)\phi u\,d\mu.$$
Since $\phi$ is arbitrary, we obtain that
\be\label{prodzero}
R_\mu(g)u=0\qquad\mu\hbox{-a.e.}
\ee
Since $u=0$ where $\mu=+\infty$, by the form of the cost functional, without loss of generality we may assume that $\O_\mu=\{u\ne0\}$. Analogously, since the cost functional can also be written as $\int_D R_\mu(g)f\,dx$, we may assume that $\mu=+\infty$ on $R_\mu(g)=0$. Thus by \eqref{prodzero} the capacitary measure $\mu$ takes only values $0$ and $+\infty$ and hence it is a domain.
\end{proof}

\subsection{Optimality condition on the boundary of the optimal sets}\label{sub:optcond}

We now formally deduce the optimality condition on the boundary of an optimal set $\O\subset D$ (for the rigorous proof we refer to \cite[Chapter 5]{hepi05}). We assume that $\O$ is sufficiently regular ($\partial\O\in C^{2,\alpha}$) and we set for simplicity $u=R_\O(f)$ and $v=R_\O(g)$. For a smooth vector field $V\in C^\infty_c(D;\R^d)$ we consider the perturbation $\O_t=(Id+tV)(\O)$ and the solutions $u_t=R_{\O_t}(f)$ and $v_t=R_{\O_t}(g)$. The formal derivatives
$$u'=\frac{d}{dt}\Big\vert_{t=0}u_t\qquad\hbox{and}\qquad v'=\frac{d}{dt}\Big\vert_{t=0}v_t$$
are solutions respectively of the problems:
$$\Delta u=0\quad\text{in }\O,\qquad u'+V\cdot\nabla u=0\quad\text{on }\partial\O;$$
$$\Delta v=0\quad\text{in }\O,\qquad v'+V\cdot\nabla v=0\quad\text{on }\partial\O.$$
Thus, the derivative of the cost functional is given by
\[\begin{split}
\frac{d}{dt}\Big\vert_{t=0}\int_{\O_t}u_t g\,dx
&=\int_\O u' g\,dx=\int_\O\nabla u'\nabla v\,dx-\int_{\partial\O}u'\frac{\partial v}{\partial n}\\
&=\int_{\partial\O}\frac{\partial v}{\partial n}V\cdot\nabla u=\int_{\partial\O}V\cdot n \frac{\partial u}{\partial n}\frac{\partial v}{\partial n}.
\end{split}\]
We now consider two cases: 
\begin{itemize}
\item If the volume constraint is saturated, that is $|\O|=1$, then we have to consider perturbations only with respect to divergence-free vector fields $V$. In this case we obtain 
$$\int_{\partial\O}V\cdot n\frac{\partial u}{\partial n}\frac{\partial v}{\partial n}=0\qquad\hbox{for every $V\in C^\infty_c(D;\R^d)$ such that }\dive V=0,$$
which gives the optimality condition 
$$\frac{\partial u}{\partial n}\frac{\partial v}{\partial n}=const\quad\text{on }\partial\O.$$
\item If the volume constraint is not saturated, that is $|\O|<1$, then we have 
$$\int_{\partial\O}V\cdot n\frac{\partial u}{\partial n}\frac{\partial v}{\partial n}=0\qquad\text{for every }V\in C^\infty_c(D;\R^d),$$
which gives the optimality condition 
$$\frac{\partial u}{\partial n}\frac{\partial v}{\partial n}=0\qquad\text{on }\partial\O.$$
In the case when $f\ge0$, we have that $|\nabla u|>0$ on the boundary of the optimal set $\O=\{u>0\}$. Thus the optimality condition can be written in the simplified form
$$\frac{\partial v}{\partial n}=0\qquad\hbox{on }\partial\O.$$
This situation is untypical for the shape optimization problem, where the cost functional is usually monotone with respect to the set inclusion. We give an explicit example of a case when the constraint is not saturated in Section \ref{s:radial}. In the next section we analyze this type of solutions and their connection with the obstacle problem. 
\end{itemize}

\section{Minimizers with nonsaturated constraint}\label{s:nonsat}

In this section we consider minimizers $\O$ which do not saturate the volume constraint, that is $|\O|<1$. We restrict our attention to the case $f\ge0$ on $D$, while the cost coefficient $g$ may change sign. Equivalently, since the resolvent operators are self-adjoint, we may consider $g\ge0$ and $f$ changing sign. In Subsection \ref{sub:contain} we prove that an optimal set $\O$ necessarily contains the set $\{g<0\}$. In Subsection \ref{sub:ostacolo} we establish a relation of the minimizer $\O$ with the obstacle problem.

\subsection{A necessary condition of optimality}\label{sub:contain}

The main result of this section is Theorem \ref{p:contain}. The argument is carried out from the point of view of the state function $u=R_\O(f)$ relative to a nonnegative right-hand side $f$. Before we pass to the statement and the proof of Theorem \ref{p:contain} we recall several classical results concerning the function $u$. 

\begin{rema}\label{oss:radon}
Let $f\in L^2(D)$ and let $u\in H^1_0(D)$ be a nonnegative function such that $\Delta u+f\ge0$ on $D$ in sense of distributions, that is 
$$\int_D\big(-\nabla u\nabla\varphi+f\varphi\big)\,dx\ge0\qquad\text{for every nonnegative }\varphi\in C^\infty_c(D).$$
It is well-known that $\nu:=\Delta u+f$ is a (positive) measure. Moreover, $\Delta u+f$ is a Radon measure in $D$. In fact, if $B_r(x_0)\subset D$, there is a nonnegative function $\varphi\in C^\infty_c(D)$ such that $\varphi=1$ on $B_r(x_0)$; thus
$$(\Delta u+f)\big(B_r(x_0)\big)\le\int_D\varphi\,d\nu=\int_D\big(-\nabla u\nabla\varphi+f\varphi\big)\,dx<+\infty.$$
\end{rema}

In what follows we use an important characterization of $\Delta u+f$ to construct competitors for the solution of the problem \eqref{relpb}. For the proof we refer to \cite{dmmu97} (Theorem 5.1).

\begin{lemma}\label{p:competitor}
Let $f\in L^2(D)$ and let $u\in H^1_0(D)$ be a nonnegative function. Then the following conditions are equivalent:
\begin{enumerate}[(i)]
\item $\Delta u+f\ind_{\{u>0\}}\ge0$ on $D$ in the sense of distributions;
\item there exists a capacitary measure $\mu\in\M_{\cp}(D)$ such that $\O_\mu=\{u>0\}$ and
$$-\Delta u+\mu u=f\quad\text{on }D,\qquad u\in H^1_\mu.$$
\end{enumerate}
\end{lemma}

Let now $\O\subset\R^d$ be a bounded quasi-open set and let $u\in H^1_0(\O)$ be the solution of
\be\label{e:u_f}
-\Delta u=f\quad\text{in }\O,\qquad u\in H^1_0(\O).
\ee
The following lemma describes the behavior of $u$ around the boundary points of low density for $\O$. The result is classical and we give the proof for the sake of completeness.

\begin{lemma}\label{l:stima_u}
Let $r_0>0$, $x_0\in\R^d$ and $f\in L^2(B_{r_0}(x_0))$, with $f\ge0$. Suppose that
$$M:=\sup_{0<r\le r_0}\bigg(\frac1{|B_r|}\int_{B_r(x_0)}f^2(x)\,dx\bigg)^{1/2}<+\infty.$$
Then there exists a constant $\eps>0$, depending only on the dimension $d$ and on $M$, such that if $\O$ satisfies the hypothesis
$$\frac{|B_r(x_0)\cap\O|}{|B_r|}\le\eps\qquad\text{for every }0<r<r_0,$$
then for the solution $u$ of \eqref{e:u_f} we have the estimate
$$\frac1{r^2|B_r|}\int_{B_r(x_0)}|\nabla u|^2\,dx\le2^{d+2}\sup\bigg\{1,\frac1{r_0^2|B_{r_0}|}\int_{B_{r_0}(x_0)}|\nabla u|^2\,dx\bigg\}.$$
\end{lemma}

\begin{proof}
Suppose, without loss of generality, that $x_0=0$. Let $0<r<r_0$ and $\phi\in C^\infty_0(B_r)$ be a function such that $0\le\phi\le1$ on $B_r$, $\phi=1$ on $B_{r/2}$ and $|\nabla\phi|\le 3/r$. The proof is obtained by iteration of the following Caccioppoli inequality:
\[\begin{split}
\int_{B_{r/2}}|\nabla u|^2\,dx&\le\int_{B_r}|\nabla(\phi u)|^2\,dx=\int_{B_r}|\nabla\phi|^2 u^2\,dx+\int_{B_r}\nabla u\nabla(\phi^2 u)\,dx\\
&=\int_{B_r}|\nabla\phi|^2 u^2\,dx+\int_{B_r}f\phi^2 u\,dx\\
&\le\frac{9}{r^2}\int_{B_r}u^2\,dx+\bigg(\int_{B_r}f^2\,dx\bigg)^{1/2}\bigg(\int_{B_r} u^2\,dx\bigg)^{1/2}.
\end{split}\]
Now, since the ball is an extension domain, there are constants $\Lambda_1>0$ and $\delta_0>0$ such that if $|\O\cap B_r|\le\delta_0|B_r|$ and $v\in H^1(B_r)$ is such that $v=0$ on $B_r\setminus\O$, then 
$$\int_{B_r}v^2\,dx\le\Lambda_1 r^2\bigg(\frac{|\O\cap B_r|}{|B_r|}\bigg)^{2/d}\int_{B_r}|\nabla v|^2\,dx.$$
Thus, we obtain,
$$\int_{B_{r/2}}|\nabla u|^2\,dx\le 9\Lambda_1\eps^{2/d}\int_{B_{r}}|\nabla u|^2\,dx+\bigg(\int_{B_{r}} f^2\,dx\bigg)^{1/2}\bigg(\Lambda_1\eps^{2/d}r^2\int_{B_{r}} |\nabla u|^2\,dx\bigg)^{1/2}.$$
Dividing by $r^2|B_r|$ we get
\[\begin{split}
\frac1{r^2|B_r|}\int_{B_{r/2}}|\nabla u|^2\,dx&\le\frac{9\Lambda_1\eps^{2/d}}{r^2|B_r|}\int_{B_r}|\nabla u|^2\,dx+\bigg(\frac1{|B_r|}\int_{B_r}f^2\,dx\bigg)^{1/2}\bigg(\frac{\Lambda_1\eps^{2/d}}{r^2|B_r|}\int_{B_r}|\nabla u|^2\,dx\bigg)^{1/2}\\
&\le\frac{9\Lambda_1\eps^{2/d}}{r^2|B_r|}\int_{B_r}|\nabla u|^2\,dx+M\Lambda_1^{1/2}\eps^{1/d}\bigg(\frac{1}{r^2|B_r|}\int_{B_r}|\nabla u|^2\,dx\bigg)^{1/2}.
\end{split}\]
Let us indicate by $r_n$ and $a_n$ the quantities
$$r_n=r_0 2^{-n},\qquad a_n=\frac{1}{r_n^2|B_{r_n}|}\int_{B_{r_n}}|\nabla u|^2\,dx.$$
Then, for $\eps$ small enough we have
$$a_{n+1}\le\frac12 a_n+\frac12 a_n^{1/2},$$
which gives that $a_n\le\sup\{1,a_0\}$, for every $n\ge1$.
\end{proof}


\begin{theo}\label{p:contain}
Let $f\ge0$, $f\in L^2(D)$ and $g\in L^2(D)$. Suppose that $\O\subset\R^d$ is a solution of the problem \eqref{minpb} such that $|\O|<1$. Then $\{fg<0\}\subset\O$.
\end{theo}

\begin{proof}
Suppose by contradiction that this is not the case. Then there is a point $x_0\in D$ such that $x_0$ is a point of density $0$ for $\O$ and $x_0$ is a Lebesgue point for $f$ and $g$ with $f(x_0>0$ and $g(x_0)<0$, that is
\[\begin{split}
&\lim_{r\to0}\frac1{|B_r|}\int_{B_r(x_0)}f(x)\,dx=f(x_0)>0,\\
&\lim_{r\to0}\frac1{|B_r|}\int_{B_r(x_0)}g(x)\,dx=g(x_0)<0,\\
&\lim_{r\to0}\frac{|\O\cap B_r(x_0)|}{|B_r|}=0.
\end{split}\]
Let $r>0$ be fixed. Consider the functions $u,v$ solutions of the problems
\[\begin{split}
&-\Delta u=f\quad\text{in }\O,\qquad u\in H^1_0(\O),\\
&-\Delta v=f\quad\text{in }B_r(x_0),\qquad v=u\quad\text{on}\quad\partial B_r(x_0),
\end{split}\]
set $\nu=\Delta u+f\ind_{\{u>0\}}$ and take $r>0$ such that $\nu(\partial B_r(x_0))=0$. The function $v-u$ is a solution of the PDE
$$-\Delta (v-u)=\nu+f\ind_{B_r(x_0)\setminus\O}\quad\text{in }B_r(x_0),\qquad v-u\in H^1_0(B_r(x_0)),$$
in the sense that for all $\psi\in H^1_0(B_r(x_0))$ we have
$$\int_{B_r(x_0)}\nabla(v-u)\nabla\psi\,dx=\int_{B_r(x_0)\setminus\O}\psi f\,dx+\int_{B_r(x_0)}\psi\,d\nu\;.$$
In particular, by the maximum principle, we have that $v-u>0$ on $B_r(x_0)$.
We now show that
\be\label{e:v-u}
\Delta (v-u)+\nu\ind_{B_r(x_0)}+f\ind_{B_r(x_0)\setminus\O}\ge0\quad\text{in}\quad D,
\ee
in sense of distributions. Let $\phi\in C^\infty_c(D)$ be a nonnegative function. For every $\eps>0$, consider the function 
$$p_\eps(t)=\begin{cases}
1&\text{if }t\ge\eps,\\
0&\text{if }t\le0,\\
t/\eps&\text{if }0\le t\le\eps.
\end{cases}$$
Then $p_\eps(v-u)\phi\in H^1_0(B_r(x_0))$ and so we have
$$\int_{B_r(x_0)}\nabla(v-u)\nabla\big(p_\eps(v-u)\phi\big)\,dx=\int_{B_r(x_0)\setminus\O}p_\eps(v-u)\phi f\,dx+\int_{B_r(x_0)}p_\eps(v-u)\phi\,d\nu,$$
which, by developing the gradient, gives
$$\int_{B_r(x_0)}p_\eps(v-u)\nabla(v-u)\nabla\phi\,dx\le\int_{B_r(x_0)\setminus\O}p_\eps(v-u)\phi f\,dx+\int_{B_r(x_0)}p_\eps(v-u)\phi\,d\nu.$$
Passing to the limit as $\eps\to0$, we obtain
$$\int_{B_r(x_0)}\nabla(v-u)\nabla\phi\,dx\le\int_{B_r(x_0)\setminus\O}\phi f\,dx+\int_{B_r(x_0)}\phi\,d\nu,$$
which concludes the proof of \eqref{e:v-u}. Define now $\tilde u\in H^1_0(\Dr)$ by
$$\tilde u(x)=\begin{cases}
u(x)&\text{if }x\in\Dr\setminus B_r(x_0),\\
v(x)&\text{if }x\in B_r(x_0).
\end{cases}$$
We aim to show that $\Delta\tilde u+f\ind_{\{\tilde u>0\}}\ge0$ on $D$. In fact, using $\phi$ as a test function for $\Delta\tilde u+f\ind_{\{\tilde u>0\}}$ we have
\[\begin{split}
\int_D(-\nabla\tilde u\nabla\phi&+f\ind_{\{\tilde u>0\}}\phi)\,dx\\
&=\int_\O(-\nabla u\nabla\phi+f\phi)\,dx+\int_{B_r}\!\!\!\big(-\nabla (v-u)\nabla\phi+f\phi\ind_{B_r(x_0)\setminus\O}\big)\,dx\\
&\ge\int_D\phi\,d\nu-\int_{B_r(x_0)}\phi\,d\nu\ge0,
\end{split}\]
which proves the claim. Thus, by Lemma \ref{p:competitor}, we have that there is a capacitary measure $\tilde\mu\in\M_{\cp}(D)$ such that $\O_{\tilde\mu}=\{\tilde u>0\}=\O\cup B_r(x_0)$ and 
$$-\Delta\tilde u+\tilde\mu\tilde u=f\quad\text{on }D,\qquad\tilde u\in H^1_{\tilde\mu}(D).$$
Now, by the optimality of $\O$ we have that for $r>0$ sufficiently small
\be\label{e:inequality_optimality}
0\le\int_D g\tilde u\,dx-\int_D g u\,dx=\int_{B_r(x_0)} g(v-u)\,dx.
\ee
In order to conclude it is now sufficient to study the asymptotic behavior of the integral on the right-hand side as $r\to0$. Assume for simplicity that $x_0=0$. We consider the functions $w$ and $h$ solutions of the equations
\[\begin{split}
&-\Delta w=f\quad\text{in }B_r,\qquad w\in H^1_0(B_r),\\
&\Delta h=0\quad\text{in }B_r,\qquad h-u\in H^1_0(B_r),
\end{split}\]
and we set 
$$f_r(x)=f(rx),\quad g_r(x)=g(rx),\quad w_{r}(x)=\frac1{r^2}w(rx),\quad h_{r}(x)=\frac1{r^2}h(rx),\quad u_{r}(x)=\frac1{r^2}u(rx).$$
We notice that:
\begin{enumerate}[(i)] 
\item Since $x_0=0$ is a Lebesgue point for both $f$ and $g$, we have that $f_r\to f(0)$ and $g_r\to g(0)$ strongly in $L^2(B_1)$, as $r\to 0$.
\item The function $w_r$ is a solution of the equation 
$$-\Delta w_r=f_r\quad\text{in }B_1,\qquad w\in H^1_0(B_1),$$
and $w_r\to w_0$ strongly in $H^1_0(B_1)$, where $w_0(x)=f(0)(1-|x|^2)/(2d)$ is the solution of 
$$-\Delta w_0=f(0)\quad\text{in }B_1,\qquad w_0\in H^1_0(B_1).$$
\item There is a constant $C$, not depending on $r$, such that 
\be\label{e:u_rest}
\int_{B_1}|\nabla h_r|^2\,dx\le\int_{B_1}|\nabla u_r|^2\,dx\le C.
\ee
The first inequality is due to the harmonicity of $h_r$, while the second one is a consequence of Lemma \ref{l:stima_u}. Thus, $\|h_r-u_r\|_{H^1}^2\le C$ and so, up to a subsequence, we may assume that $z_r=h_r-u_r$ converges weakly in $H^1_0(B_1)$ and strongly in $L^2(B_1)$ to some function $z_0\in H^1_0(B_1)$. We now prove that $z_0=0$. In fact, given a function $\phi\in C^\infty_c(B_1)$ we have that
\[\begin{split}
\int_{B_1}\nabla\phi\nabla z_r\,dx
&=-\int_{B_1}\nabla\phi\nabla u_r\,dx\\
&\le\|\nabla\phi\|_{L^\infty}|\{u_r\ne0\}\cap B_1|^{1/2}\bigg(\int_{B_1}|\nabla u_r|^2\,dx\bigg)^{1/2}\\
&\le C\|\nabla\phi\|_{L^\infty}\left(\frac{|\O\cap B_r|}{|B_r|}\right)^{1/2},
\end{split}\]
where the equality is due to the fact that $h_r$ is harmonic, the first inequality is by Cauchy-Schwartz, and the last inequality is due to the estimate \eqref{e:u_rest}. Now since the density of $\O$ is zero in $0$, passing to the limit as $r\to 0$, we obtain
$$\int_{B_1}\nabla\phi\nabla z_0\,dx\le0.$$
Since $\phi$ is arbitrary we obtain that $z_0$ is harmonic in $B_1$ and since $z_0=0$ on $\partial B_1$ we get that $z_0=0$. Thus we conclude that
$$\lim_{r\to0}\int_{B_1}|h_r-u_r|^2\,dx=0.$$
\end{enumerate}
By the results from (i), (ii) and (iii), we get that
$$\int_{B_r}g(v-u)\,dx=r^{2-d}\int_{B_1}g_r(w_r+h_r-u_r)\,dx= r^{2-d}\bigg(\int_{B_1}g(x_0)w_0(x)\,dx+o(r)\bigg),$$
which is strictly negative, for $r>0$ sufficiently small, so contradicting \eqref{e:inequality_optimality}.
\end{proof}

\begin{rema}
Since the resolvent operator is self-adjoint, in Theorem \ref{p:contain} we may equivalently assume $g\ge0$ and deduce that if $|\O|<1$ then $\{gf<0\}\subset\O$. By a simple change of sign in the data we also have that if $f\le0$ (or if $g\le0$) and $|\O|<1$, then $\{gf<0\}\subset\O$.
\end{rema}

\section{Unconstrained minimizers and the obstacle problem}\label{sub:ostacolo}

Let $D\subset\R^d$ be a bounded open set. We say that $\O\subset D$ is an unconstrained minimizer if it is a solution of the optimization problem 
\be\label{e:unconstrained}
\min\Big\{\int_\O R_\O(g) f(x)\,dx\ :\ \O\hbox{ quasi-open, }\O\subset D\Big\}
\ee
where we removed the measure constraint on $\O$. In Proposition \ref{p:ostacolo} we prove that the solution of \eqref{e:unconstrained} is related to the solution of the obstacle problem
\be\label{e:ostacolo}
\min\Big\{\frac12\int_{D}|\nabla v|^2\,dx+\int_D g(x)v(x)\,dx \ :\ v\in H^1_0(D),\ v\ge0\hbox{ on }D\Big\}. 
\ee
We first prove the following lemma characterizing the solutions of \eqref{e:ostacolo}. 
\begin{lemma}\label{l:ostacolo}
Let $D\subset\R^d$ be a bounded open set and $g\in L^2(D)$. Then the solution $v$ of the obstacle problem \eqref{e:ostacolo} satisfies 
\be\label{e:supv}
v=\sup_{\O\subset D} v_\O,
\ee
where the maximum is over all quasi-open subsets $\O\subset D$ and $v_\O$ is the solution of
\be\label{e:v_Omega}
\Delta v_\O=g\quad\hbox{in }\O,\qquad v_\O\in H^1_0(\O).
\ee
\end{lemma}

\begin{proof}
Suppose that $\O\subset D$ is a quasi-open set. It is sufficient to prove that $v\ge v_\O$ in $D$. Indeed, set
$$J(u)=\frac12\int_D|\nabla u|^2\,dx+\int_Du(x)g(x)\,dx$$
and consider the test functions $v\vee v_\O$ and $v\wedge v_\O$. Since $v\vee v_\O\le0$ in $D$ and $v\wedge v_\O\in H^1_0(\O)$, we have the inequalities
$$J(v)\le J(v\vee v_\O)\qquad\hbox{and}\qquad J(v_\O)\le J(v\wedge v_\O).$$
On the other hand, by the definition of $J$ we have
$$J(v)+J(v_\O)=J(v\vee v_\O)+J(v\wedge v_\O).$$
Thus, we obtain
$$J(v)= J(v\vee v_\O)\qquad\text{and}\qquad J(v_\O)= J(v\wedge v_\O).$$
By the uniqueness of the solution of the obstacle problem and of the equation \eqref{e:v_Omega}, we have that $v=v\vee v_\O$ and $v_\O=v\wedge v_\O$ which concludes the proof. 
\end{proof}

\begin{rema}
The supremum in \eqref{e:supv} is realized by the quasi-open set $\O=\{v>0\}$.
\end{rema}

\begin{rema}
By the density of the (smooth) open sets in the family of quasi-open sets we have that
$$v=\sup\big\{v_\O\ :\ \O\hbox{ open, }\O\subset D\big\}.$$
\end{rema}

\begin{prop}\label{p:ostacolo}
Let $D\subset\R^d$ be a bounded open set and let $f,g\in L^2(D)$ with $f\ge0$ on $D$. Then the unique minimizer of the unconstrained problem \eqref{e:unconstrained} is the quasi-open set $\O=\{v>0\}$, where $v$ is the solution of the obstacle problem \eqref{e:ostacolo}.
\end{prop}

\begin{proof}
Let $\O\subset D$ be a quasi-open set. By Lemma \ref{l:ostacolo} we have that $v\ge v_\O$. Since $f\ge0$ we have that
$$\int_D R_\O(g)f\,dx=-\int_D v_\O f\,dx\ge-\int_D vf\,dx=\int_D R_{\{v>0\}}(g)f\,dx,$$
which concludes the proof.
\end{proof}

As a corollary we obtain the following result.

\begin{cor}\label{cor:ostacolo}
Let $D\subset\R^d$ be a bounded open set and let $f,g\in L^2(D)$ with $f\ge0$ in $D$. Suppose that $\O\subset\Dr$ is a solution of the optimization problem \eqref{minpb} such that $|\O|<1$. Then:
\begin{enumerate}[(i)]
\item if $g\in L^p(D)$, for some $p>d$, then $\O$ is an open subset of $D$ and the function $v=R_\O(g)$ is $C^{1,\beta}$ regular in $D$, where $\beta=1-d/p$;
\item if the set $\{g>0\}$ is open and $g:\{g>0\}\to\R_+$ is H\"older continuous, then $v$ is $C^{1,1}$ regular in the set $\{g>0\}$ and $|\nabla v|=0$ on the free boundary $\partial\O\cap\{g>0\}$;
\item under the hypotheses from the previous point, the free boundary $\partial\O\cap\{g>0\}$ can be decomposed into two disjoint sets $\text{Reg}\,(\partial\O)$ and $\text{Sing}\,(\partial\O)$, where:
\begin{itemize}
\item $\text{Reg}\,(\partial\O)$ is an open subset of $\partial\O\cap\{g>0\}$ and is locally the graph of a $C^{1,\alpha}$ function, for some $\alpha>0$; if $g\in C^\infty(\{g>0\})$, then $\text{Reg}\,(\partial\O)$ is smooth;
\item $\text{Sing}\,(\partial\O)$ is contained in a countable union of $(d-1)$-dimensional manifolds.
\end{itemize} 
\end{enumerate}
\end{cor}

\begin{proof}
We first notice that since $\O$ is such that $|\O|<1$, it is an unconstrained minimizer of \eqref{e:unconstrained} in the set $\tilde D=\O\cup B_r(x_0)\cap D$, for every sufficiently small ball $B_r(x_0)$. By Proposition \ref{p:ostacolo}, the function $R_\O(-g)$ is a solution of the obstacle problem \eqref{e:ostacolo} in $\tilde D$. Thus, all the regularity result for the obstacle problem are valid for $v=R_\O(g)$, in particular the statements {\it(i)}, {\it(ii)} and {\it(iii)}. For the proof of {\it(i)} we refer to \cite{bresta}, while for {\it(ii)} and {\it(iii)}, we refer to \cite{caf}, \cite{kinnir} and \cite{weiss}.
\end{proof}

\section{The case of radially symmetric cost functional}\label{s:radial}

In this section we consider a special class of functionals, where $f=1$ and $g:\R^d\to\R$ is radially symmetric and nondecreasing on each radius. It is natural to conjecture that in this situation the optimal set is a ball centered at the origin. In the case when $g\le0$ this follows by a classical symmetrization argument; on the other hand, if $g$ changes sign, the cost functional is nonmonotone and the known symmetrization results fail in the comparison argument of a general domain with a ball of the same measure. In this section we prove the following proposition.

\begin{prop}\label{radprop}
Suppose that $f=1$ and $g:\R^d\to\R$ is a given radially symmetric nondecreasing function such that $g(0)<0$. Then, setting
$$R_g=\sup\bigg\{R>0\ :\ \int_0^R r^{d-1}g(r)\,dr\le0\bigg\},$$
the ball centered at the origin of radius $\inf\{\omega_d^{-1/d},R_g\}$ is a solution of the problem
\be\label{radialopt}
\min\bigg\{\int_D g(x)u_\O(x)\,dx\ :\ \O\subset\R^d,\ \O\hbox{ quasi open, }|\O|\le1\bigg\}.
\ee
\end{prop}

\begin{rema}
The condition $g(0)<0$ assures that the solution of \eqref{radialopt} is nontrivial. Indeed, if $g\ge 0$ on $\R^d$, then the empty set is a solution as well as every quasi-open subset of $\{g=0\}$.
\end{rema}

As a consequence of Proposition \ref{radprop} we obtain the following example.

\begin{exam}
Suppose that $f=1$ and $g=-\ind_{B_{r_0}}+\ind_{B_{r_0}^c}$ for some radius $r_0>0$. Then the solution $\O_{opt}$ of the problem \eqref{radialopt} is unique and is given by the ball of volume $\min\{2|B_{r_0}|,1\}$. Indeed, the solution is a ball $B_R$ that contains the set $B_{r_0}=\{g<0\}$. The energy of the ball $B_R$ is given by the formula
$$f(R)=d\omega_d\bigg[-\int_0^{r_0}\frac{R^2-r^2}{2d}r^{d-1}\,dr+\int_{r_0}^R\frac{R^2-r^2}{2d}r^{d-1}\,dr\bigg].$$
Taking the derivative with respect to $R$ we get that
$$f'(R)=\frac{\omega_d}{d}\Big[R(R^d-2r_0^d)\Big].$$
Thus, the function $f$ achieves its minimum at $2^{1/d}r_0$, if $2r_0^d\le\omega_d^{-1/d}$ and at $1$, if $2r_0^d\ge\omega_d^{-1/d}$, which gives the claim.
\end{exam}

The rest of the section is dedicated to the proof of Proposition \ref{radprop}.

\subsection{The Schwarz rearrangement of a torsion function}

Let $\O\subset\R^d$ be a bounded open or quasi-open set and $u\in H^1_0(\O)$ be the torsion function of $\O$, that is the solution of the problem
\be\label{equ1}
-\Delta u=1\quad\hbox{in }\O,\qquad u\in H^1_0(\O).
\ee
Let $\O^*$ be the ball centered at zero of measure $|\O|$ and let $u^*:\O^*\to\R$ be the radially decreasing rearrangement of $u$. We set $M=\|u\|_{L^\infty(\O)}$ and $\O_t=\{u>t\}$, for every $t\in[0,M]$. Then the set $\O^*_t=\{u^*>t\}$ is the ball centered at zero of measure $|\O^*_t|=|\O_t|$. On every set $\O_t^*$ we consider the function $w_t$ solution of the PDE 
$$-\Delta w_t=1\quad\hbox{in }\O_t^*,\qquad w_t\in H^1_0(\O_t^*).$$
The well-known result of Talenti \cite{talenti} gives that
$$u^*(x)-t\le w_t(x)\qquad\text{for every $x\in\O^*_t$ and every }t\in[0,M].$$
In the next lemma we use this comparison to obtain that the function $u^*$ is itself a solution of a certain PDE on $\O^*$. 

\begin{lemma}\label{l:subharm}
Let $\O\subset\R^d$ be a bounded quasi-open set and let $u$ be the solution of \eqref{equ1}. Then the Steiner symmetrization $u^*$ of $u$ is a solution of the equation
\be\label{equstar}
-\dive\big((1+a(u^*))\nabla u^*\big)=1\quad\text{in }\O^*,\qquad u^*\in H^1_0(\O^*),
\ee
where $a:[0,M]\to \R$ is a nonnegative function. 
\end{lemma}

\begin{proof}
We use the notation introduced at the begining of the section. Let $f:[0,M]\to\R$ be a given $C^1$ function such that $f(0)=0$. Then $f(u^*)\in H^1_0(\O^*)$ and we have
\be\label{thndronthemntn}
\begin{split}
\int_{\O^*}\nabla u^*\nabla f(u^*)\,dx-\int_{\O^*}f(u^*)\,dx
&=\int_{\O^*}f'(u^*)|\nabla u^*|^2\,dx-\int_\O f(u)\,dx\\
&=\int_0^M f'(t)\int_{\partial\O_t^*}|\nabla u^*|\,d\HH^{d-1}\,dt-\int_\O f(u)\,dx\\
&=-\int_0^M f'(t)a(t)\int_{\partial\O_t^*}|\nabla u^*|\,d\HH^{d-1}\,dt\\
&\quad+\int_0^M f'(t)\int_{\partial\O_t^*}|\nabla w_t|\,d\HH^{d-1}\,dt-\int_{\O} f(u)\,dx,
\end{split}\ee
where we set 
$$a(t):=\frac{\ds\int_{\partial\O_t^*}|\nabla w_t|\,d\HH^{d-1}-\int_{\partial\O_t^*}|\nabla u^*|\,d\HH^{d-1}}{\ds\int_{\partial\O_t^*}|\nabla u^*|\,d\HH^{d-1}}.$$
We now notice that the difference of the last two terms in \eqref{thndronthemntn} vanishes. Indeed, using an integration by parts for $w_t$ we get 
$$\int_{\partial\O_t^*}|\nabla w_t|\,d\HH^{d-1}=-\int_{\partial\O_t^*}\frac{\partial w_t}{\partial n}\,d\HH^{d-1}=-\int_{\O_t^*}\Delta w_t\,dx=|\O_t^*|.$$
Analogously, since $u-t$ is the solution of $-\Delta (u-t)=1$ on $\O_t$ we get 
$$\int_{\partial\O_t}|\nabla u|\,d\HH^{d-1}=-\int_{\partial\O_t}\frac{\partial u}{\partial n}\,d\HH^{d-1}=-\int_{\O_t}\Delta u\,dx=|\O_t|.$$
Since $|\O_t|=|\O_t^*|$ we obtain 
\[\begin{split}
\int_0^M f'(t)\int_{\partial\O_t^*}|\nabla w_t|\,d\HH^{d-1}\,dt
&=\int_0^M f'(t)\int_{\partial\O_t}|\nabla u|\,d\HH^{d-1}\,dt\\
&=\int_{\O_t}f'(u)|\nabla u|^2\,dx=\int_\O f(u)\,dx.
\end{split}\]
On the other hand, by the co-area formula, the first term in the last line of \eqref{thndronthemntn} can be rewritten as
\[\begin{split}
\int_0^M f'(t)a(t)\int_{\partial\O_t^*}|\nabla u^*|\,d\HH^{d-1}\,dt
&=\int_{\O^*}f'(u^*)a(u^*)|\nabla u^*|^2\,dx\\
&=\int_{\O^*}a(u^*)\nabla u^*\nabla f(u^*)\,dx.
\end{split}\]
Thus, by \eqref{thndronthemntn} we infer
$$\int_{\O^*}(1+a(u^*))\nabla u^*\nabla f(u^*)\,dx=\int_{\O^*} f(u^*)\,dx.$$
Since the equality is true for every $f$, with $f(0)=0$, we obtain that $u^*$ is a solution of \eqref{equstar}.
\end{proof}

In the next subsection we establish which is the optimal function $a$ on a ball of fixed radius $R$.

\subsection{An optimization problem for radially decreasing functions}

Let $a:[0,R_0]\to [0,+\infty)$ be a given nonnegative measurable function. Let $R\le R_0$ and $u_{a,R}$ be the solution of the PDE
$$-\dive\big((1+a)\nabla u\big)=1\quad\text{in }B_R,\qquad u\in H^1_0(B_R).$$
Then $u_{a,R}=u_{a,R}(r)$ is radially symmetric and is a solution of the problem
$$-\frac{1}{r^{d-1}}\partial_r\Big(r^{d-1}(1+a(r))\partial_r u(r)\Big)=1\quad\text{in }(0,R),\qquad u(R)=u'(0)=0.$$
Integrating in $r$ we get that $u_{a,R}$ is explicitly given by 
$$u_{a,R}(r)=\frac1d\int_r^R\frac{s}{1+a(s)}\,ds. $$
We consider a radial nondecreasing function $g:\R^d\to\R$ such that $g(0)<0$ and the associated cost functional $\F(a,R)$ given by 
$$\F(a,R)=\int_{B_R}g(x)u_{a,R}(x)\,dx.$$
Setting
$$G(s)=\int_0^s r^{d-1}g(r)\,dr$$
we obtain
$$\F(a,R)=\frac1d\int_0^R r^{d-1}g(r)\int_r^R\frac{s}{1+a(s)}\,ds\,dr=\frac1d\int_0^R \frac{G(s)\,s}{1+a(s)}\,ds.$$
Since $g$ is nondecreasing and $g(0)<0$, we have that the set $\{g\le 0\}$ is an interval of the form $[0,R_g]$ (we set $R_g=+\infty$ in the case when $g\le 0$ on $\R^d$). Then we have
\be\label{aRinequality}
\begin{cases}
\F(a,R)\ge\F(0,R)&\text{if }R\le R_g,\\
\F(a,R)\ge\F(0,R_g)&\text{if }R>R_g.\\
\end{cases}
\ee
Indeed, if $R\le R_g$. Then $G\le0$ and \eqref{aRinequality} follows since in this case the functional $\F(a,R)$ is monotone increasing in $a$. On the other hand, if $R>R_g$, we have that 
$$\F(a,R)=\frac1d\int_0^R\frac{G(s)\,s}{1+a(s)}\,ds\ge\frac1d\int_0^{R_g}\frac{G(s)\,s}{1+a(s)}\,ds=\F(a,R_g),$$
and \eqref{aRinequality} again follows since $\F(a,R_g)$ is monotone increasing in $a$.

\begin{proof}[Proof of Proposition \ref{radprop}] Given a quasi-open set $\O\subset\R^d$ and a function $u$ solution of \eqref{equ1} we consider the ball $\O^*$ of measure $\O$ and the symmetrized function $u^*$. By the Riesz inequality we have that 
$$\int_\O g(x)u(x)\,dx\ge\int_{\O^*}g^*(x)u^*(x)\,dx.$$
By Lemma \ref{l:subharm} we get that
$$\int_{\O^*}g^*(x)u^*(x)\,dx=\F\big(a(u^*),R\big),$$
where $R$ is the radius of $\O^*$. Now the inequality \eqref{aRinequality} gives that
$$\F(a(u^*),R)\ge\F(0,R\wedge R_g)\ge\F(0,\omega_d^{-1/d}\wedge R_g).$$
If $B$ is the ball of radius $\omega_d^{-1/d}\wedge R_g$, by the definition of $\F$ we have that
$$\F(0,R\wedge R_g)=\int_B g(x)u_B(x)\,dx,$$
which concludes the proof of Proposition \ref{radprop}.
\end{proof}

\bigskip
\ack The work of the first author is part of the project 2015PA5MP7 {\it``Calcolo delle Variazioni''} funded by the Italian Ministry of Research and University. The first author is member of the Gruppo Nazionale per l'Analisi Matematica, la Probabilit\`a e le loro Applicazioni (GNAMPA) of the Istituto Nazionale di Alta Matematica (INdAM). The second author has been partially supported by the LabEx PERSYVAL-Lab (ANR-11-LABX-0025-01) project GeoSpec and the project ANR COMEDIC.


\bigskip
{\small\noindent
Giuseppe Buttazzo:
Dipartimento di Matematica,
Universit\`a di Pisa\\
Largo B. Pontecorvo 5,
56127 Pisa - ITALY\\
{\tt buttazzo@dm.unipi.it}\\
{\tt http://www.dm.unipi.it/pages/buttazzo/}

\bigskip\noindent
{Bozhidar Velichkov:
Laboratoire Jean Kuntzmann (LJK), 
Universit\'e Grenoble Alpes\\
B\^atiment IMAG, 700 Avenue Centrale, 
38401 Saint-Martin-d'H\`eres - FRANCE}\\
{\tt bozhidar.velichkov@univ-grenoble-alpes.fr}\\
{\tt http://www.velichkov.it}

\end{document}